\documentclass[12pt,a4paper]{amsart}
 \usepackage[T1]{fontenc}
\usepackage[top=35mm, bottom=35mm, left=30mm, right=30mm]{geometry}	
\usepackage{amssymb}
\usepackage{mathtools}
\usepackage{verbatim,enumerate}
\usepackage{xcolor}

\usepackage[looser]{newpxtext}
\usepackage[smallerops]{newpxmath}
\usepackage{fancyhdr}
\makeatletter
\def\@@and{\MakeLowercase{and}}
\makeatother

\usepackage[hidelinks,pagebackref]{hyperref}

\usepackage[nobysame,msc-links]{amsrefs} 

\theoremstyle{definition}
\newtheorem{defn}{Definition}[section]
\newtheorem{exam}[defn]{Example}

\newtheorem{rem}[defn]{Remark}

\theoremstyle{plain}
\newtheorem{thm}[defn]{Theorem}

\newtheorem{prop}[defn]{Proposition}
\newtheorem{coro}[defn]{Corollary}

\DeclareMathOperator{\udens}{\overline{dens}}
\DeclareMathOperator{\ldens}{\underline{dens}}

\title[T\MakeLowercase{richotomy for the orbits of a hypercyclic operator  on a }B\MakeLowercase{anach space}] 
{T\MakeLowercase{richotomy for the orbits of a hypercyclic operator on a }B\MakeLowercase{anach space}}

\author[J. L\MakeLowercase{i}]{J\MakeLowercase{ian} L\MakeLowercase{i}}
\address[J. Li]{Department of Mathematics,
Shantou University, Shantou, 515821, Guangdong, China}
\email{lijian09@mail.ustc.edu.cn}
\urladdr{https://orcid.org/0000-0002-8724-3050}

	

\date{\today}

\begin{document}

\begin{abstract}
	We obtain a trichotomy for the orbits of a hypercyclic operator $T$ on a separable Banach space $X$:
	(1) every vector is mean asymptotic to zero;
	(2) generic vectors are absolutely mean irregular;
	(3) every  hypercyclic vector is mean divergent to infinity.
	Examples of weighted backward shifts on $\ell^p$ show that all three cases can happen.
\end{abstract}

\maketitle


\section{Introduction}
Let $(X,\Vert\,\cdot\,\Vert)$ be a separable Banach space  and $T\colon X\to X$ be a bounded linear operator.
For a vector $x\in X$, the orbit of $x$ is the set $\{T^nx\colon x\geq 0\}$, where $T^n$ is the $n$th composition of $T$.
The study of the behaviors of the orbits of an operator has attracted a lot of attention.
Particularly, the Chapter III of the book \cite{B1988} is entitled \emph{The Orbit of a Linear Operator}.
According to \cite{B1988}*{Chapter III}, by ``regular'' orbits, we mean for instance that the orbit may tend to infinity, or may tend to zero, or may stay inside a ball centered at $0$, or may stay outside a ball centered at $0$.
Thanks to the Jordan decomposition theorem, every orbit for an operator on a finite-dimensional space is regular, see e.g.\@ \cite{GP2011}*{Proposition 2.57}.
The opposite case is the \emph{irregular orbits}, that is which satisfy
\[
	\liminf_{n\to\infty}\Vert T^nx\Vert =0 \text{ and } \limsup_{n\to\infty}\Vert T^nx\Vert =\infty.
\]
A special case of irregular orbits are those of hypercyclic vectors.
A vector $x\in X$ is called a \emph{hypercyclic vector} if the orbit of $x$ is dense in the whole space,
and an operator $T$ is called \emph{hypercyclic} if there is some hypercyclic vector in $X$.
In fact, a hypercyclic vector is called a transitive point in the context of topological dynamics.
By the Birkhoff transitivity theorem (see e.g. \cite{GP2011}*{Theorem 1.16}), if $T$ is hypercyclic then the collection of hypercyclic vectors is a dense $G_\delta$ subset of $X$.

In \cite{BBMP2011}*{Theorem 5}, Berm\'udez et al.\@ characterized Li-Yorke chaos in terms of the existence of irregular vectors.
Recall that an operator $T\colon X\to X$ is called \emph{Li-Yorke chaotic}
if there exists an uncountable subset $S$ of $X$ such that every pair $x,y\in S$ of distinct points we have
\[
	\liminf_{n\to\infty}\Vert T^nx-T^ny\Vert =0 \text{ and } \limsup_{n\to\infty}\Vert T^nx-T^ny\Vert >0.
\]
In the study of distributional chaos of operators on Banach spaces \cite{BBPW2018}, Bernardes et al.\@ introduced the concept of absolutely mean irregular vectors,
that is a vector $x\in X$ is an \emph{absolutely mean irregular} if
\[
	\liminf_{n\to\infty} \frac{1}{n}\sum_{i=1}^n \Vert T^ix\Vert =0 \text{ and } \limsup_{n\to\infty}\frac{1}{n}\sum_{i=1}^n \Vert T^ix\Vert =\infty.
\]
Furthermore, in \cite{BBP2020} Bernardes et al.\@ characterized mean Li-Yorke chaos in terms of the existence of absolutely mean  irregular vectors,
and characterized dense mean Li-Yorke chaos in terms of the existence of a dense (or residual) set of absolutely mean irregular vectors.
Recall that an operator $T\colon X\to X$ is called \emph{(dense) mean Li-Yorke chaotic}
if there exists a (dense) uncountable subset $S$ of $X$ such that every pair $x,y\in S$ of distinct points we have
\[
	\liminf_{n\to\infty}\frac{1}{n}\sum_{i=1}^n \Vert T^ix-T^iy\Vert =0 \text{ and } \limsup_{n\to\infty}\frac{1}{n}\sum_{i=1}^n \Vert T^ix-T^iy\Vert>0.
\]

In this paper, we study behaviors of orbits in the mean sense of a hypercyclic operator and
obtain the following trichotomy.

\begin{thm}\label{thm:trich-hypercity}
	Let $T$ be a bounded linear operator on a separable Banach space $X$.
	If $T$ is hypercyclic, then exactly one of three assertions holds:
	\begin{enumerate}
		\item  every vector is mean asymptotic to zero, that is for every $x\in X$,
		      \[
			      \lim_{n\to\infty}\frac{1}{n}\sum_{i=1}^{n}\Vert T^ix\Vert=0;
		      \]
		\item generic vectors are absolutely mean irregular, that is there exists a residual subset $X_0$ of $X$ such that for every $x\in X_0$,
		      \[
			      \liminf_{n\to\infty}\frac{1}{n}\sum_{i=1}^{n}\Vert T^ix\Vert=0
		      \]
		      and
		      \[
			      \limsup_{n\to\infty} \frac{1}{n}\sum_{i=1}^n \Vert T^i x\Vert =\infty;
		      \]
		\item every hypercyclic vector is mean divergent to infinity, that is  for every hypercyclic vector $x\in X$, 
		      \[
			      \lim_{n\to\infty} \frac{1}{n}\sum_{i=1}^n \Vert T^i x\Vert =\infty.
		      \]
	\end{enumerate}
\end{thm}

Let $\mathcal{L}(X)$ be the collection of all bounded linear operators from $X$ to itself.
The \emph{strong operator topology} on $\mathcal{L}(X)$
is defined as follows: every $T\in \mathcal{L}(X)$ has a neighborhood basis consisting of sets of the form
\[
	V_{x_1,\dotsc,x_n,\varepsilon,T}=\{
	S\in \mathcal{L}(X)\colon \|Sx_i-Tx_i\|<\varepsilon,\ i=1,\dotsc,n\},
\]
where $x_1,\dotsc,x_n\in X$ and $\varepsilon>0$.
It is not easy to see that, the Rolewicz's operator, $T\colon \ell^p(\mathbb{N})\to \ell^p(\mathbb{N})$, $(x_1,x_2,\dots)\mapsto \lambda(x_2,x_3,\dotsc)$ with $|\lambda|>1$, has a residual set of absolutely mean irregular vectors, see e.g. \cite{BBP2020}*{Corollary 30} or Corollary~\ref{coro:Kitai-criterion} in Section 3.
The following result reveals that there are plenty of hypercyclic operators with a residual set of  absolutely mean irregular vectors.

\begin{thm}\label{thm:SOT-dense}
	If $X$ is an infinite-dimensional separable Banach space,
	then the collection of hypercyclic operators with a residual set of  absolutely mean irregular vectors  is dense in  $\mathcal{L}(X)$ with respect to the strong operator topology.
\end{thm}

It is shown in \cite{BBP2020}*{Theorem 17} that
if  $T\colon X\to X$ is a bounded linear operator on a separable Banach space $X$ then $T$ is a densely mean Li-Yorke chaotic if and only if it has a residual set of absolutely mean irregular vectors. 
So we have the following consequence.
\begin{coro}
Every infinite-dimensional separable Banach space admits a dense mean Li-Yorke chaotic operator.
\end{coro}

Theorems \ref{thm:trich-hypercity} and~\ref{thm:SOT-dense} are proved in Sections 2 and 3 respectively.
In Section 3 we also strengthen Theorem~\ref{thm:SOT-dense} a litter bit for operators on a Hilbert space, see Theorem~\ref{thm:Hilbert-SOT-star},
and present some examples of weighted backward shifts on $\ell^p$ for the three cases in Theorem~\ref{thm:trich-hypercity}.


\section{Proof of Theorem~\ref{thm:trich-hypercity}}

The aim of this section is to prove Theorem~\ref{thm:trich-hypercity}.
To this end, we need some preparations.

For a (general) dynamical system, we mean a pair $(X,T)$, where $(X,d)$ is a metric space and $T\colon X\to X$ is a continuous map.
Following~\cite{LTY2015}, we say that a dynamical system $(X,T)$ is \emph{mean equicontinuous}
if for every $\varepsilon>0$ there exists $\delta>0$ such that
for every $x,y\in X$ with $d(x,y)<\delta$,
\[
	\limsup_{n\to\infty}\frac{1}{n}\sum_{i=1}^{n}d(T^ix,T^iy)<\varepsilon,
\]
and
\emph{mean sensitive}
if there exists $\delta>0$ such that for non-empty open subset $U$ of $X$, 
there exist two points $x$ and $y$ in $U$ satisfying 
\[
	\limsup_{n\to\infty}\frac{1}{n}\sum_{i=1}^{n}d(T^ix,T^iy)\geq \delta.
\]

It is show in \cite{LTY2015}*{Corollary 5.5} that for a dynamical system on a compact metric space, if it is minimal then it either mean equicontinuous or mean sensitive. 
The authors in \cite{JL2024} studied mean equicontinutity and mean sensitivity for endomorphisms of completely metrizable groups and obtained the following dichotomy.

\begin{thm}[\cite{JL2024}*{Theorem 4.9}] \label {thm:dich-mean-eq}
	Let $T$ be a bounded linear operator on a Banach space $X$.
	Then either $T$ is mean equicontinuous or mean sensitive.
\end{thm}

Combining \cite{JL2024}*{Lemma 4.5 and Proposition 4.33}, we have the following result.

\begin{prop}\label{prop:mean-sen-eq}
	Let $T$ be a bounded linear operator on a Banach space $X$.
	Then $T$ is mean sensitive if and only if there exists a residual subset $X_0$ of $X$ such that for every $x\in X_0$,
	\[
		\limsup_{n\to\infty} \frac{1}{n}\sum_{i=1}^n \Vert T^i x\Vert =\infty.
	\]
\end{prop}

We have the following property about the orbits of mean equicontinuous operators.
\begin{prop}\label{prop:mean-eq}
Let $T$ be a  bounded linear operator on a Banach space $X$.
Assume that $T$ is mean equicontinuous.
\begin{enumerate}
    \item For a vector $x\in X$, if $\liminf_{n\to\infty}\Vert T^nx \Vert =0$,
    then 
    \[
		\lim_{n\to\infty}\frac{1}{n}\sum_{i=1}^{n}\Vert T^ix\Vert=0.
	\]
    \item The collection of vectors $x$ in $X$ satisfying 
       \[
		\lim_{n\to\infty}\frac{1}{n}\sum_{i=1}^{n}\Vert T^ix\Vert=0.
	\]
 is a closed subspace of $X$.
\end{enumerate}
\end{prop}
\begin{proof}
As $T$ is mean equicontinuous,
for every $\varepsilon>0$ there exists $\delta>0$
	such that for any $u,v\in X$ with $\Vert u-v\Vert <\delta$,
	\[
		\limsup_{n\to\infty}  \frac{1}{n}\sum_{i=1}^{n}\Vert T^i(u-v)\Vert<\varepsilon.
	\]
(1) 
Fix a vector $x\in X$ with $\liminf_{n\to\infty}\Vert T^nx \Vert =0$.
For the above $\delta$, there exists $q\in\mathbb{N}$ such that $\Vert T^qx\Vert <\delta$.
	Then
	\[
		\limsup_{n\to\infty}  \frac{1}{n}\sum_{i=1}^{n}\Vert T^i(T^qx)\Vert<\varepsilon.
	\]
	and
	\begin{align*}
		\limsup_{n\to\infty}  \frac{1}{n}\sum_{i=1}^{n}\Vert T^ix \Vert & = \limsup_{n\to\infty}  \frac{1}{n+q}\sum_{i=1}^{n+q}\Vert T^ix \Vert   \\
	& = \limsup_{n\to\infty}  \biggl (\frac{n}{n+q} \frac{1}{n} \sum_{i=1}^{n}\Vert T^i(T^qx) \Vert + \frac{1}{n+q}\sum_{i=1}^q  \Vert T^ix\Vert \biggr) \\
	 & <\varepsilon.
	\end{align*}
	By the arbitrariness of $\varepsilon$,
	one has
	\[
		\lim_{n\to\infty}  \frac{1}{n}\sum_{i=1}^{n}\Vert T^ix\Vert=0.
	\]

(2) By the triangle inequality of norm, it is easy to see that 
 \[
	X_0=\Bigl \{x\in X\colon \lim_{n\to\infty}\frac{1}{n}\sum_{i=1}^{n}\Vert T^ix\Vert=0 \Bigr\}
	\]
 is a subspace of $X$. 
Let $y\in X $ be an accumulation point of $X_0$.
For the above $\delta$, 
there exists a vector $z\in X_0$ such that $\Vert z-y\Vert <\delta$.
	Then
	\begin{align*}
		\limsup_{n\to\infty}   \frac{1}{n}\sum_{i=1}^{n}\Vert T^iy\Vert \leq
		\limsup_{n\to\infty}  \frac{1}{n}\sum_{i=1}^{n} (\Vert T^i(z-y) \Vert + \Vert T^i z\Vert )
		< \varepsilon.
	\end{align*}
	By the arbitrariness of $\varepsilon$ again,
	one has
	\[
		\lim_{n\to\infty} \frac{1}{n}\sum_{i=1}^{n}\Vert T^iy\Vert=0.
	\]
 This shows that $X_0$ is closed.
\end{proof}

We divide the proof of Theorem~\ref{thm:trich-hypercity}
into two parts.
First by the dichotomy result of Theorem~\ref{thm:dich-mean-eq},
we have the following dichotomy on the orbits, which implies the case (1) in Theorem~\ref{thm:trich-hypercity} and its opposite side.

\begin{prop}\label{prop:dict-mean-asy}
	Let $T$ be a bounded linear operator on a separable Banach space $X$.
	If $T$ is hypercyclic, then either for every $x\in X$,
	\[
		\lim_{n\to\infty}\frac{1}{n}\sum_{i=1}^{n}\Vert T^ix\Vert=0
	\]
	or there exists a residual subset $X_0$ of $X$ such that for every $x\in X_0$
	\[
		\limsup_{n\to\infty} \frac{1}{n}\sum_{i=1}^n \Vert T^i x\Vert =\infty.
	\]
\end{prop}

\begin{proof}
	By Theorem~\ref{thm:dich-mean-eq}, $T$ is either mean equicontinuous or mean sensitive.
	First assume that $T$ is mean equicontinuous.
    For every hypercyclic vector $x\in X$, one has $\liminf_{n\to\infty}\Vert T^nx \Vert =0$. 
    Then by Proposition~\ref{prop:mean-eq}~(1), 
    \[
		\lim_{n\to\infty} \frac{1}{n}\sum_{i=1}^{n}\Vert T^ix\Vert=0.
	\]
 As the collection of hypercyclic vectors is dense in $X$,
 by Proposition~\ref{prop:mean-eq}~(2) 	one has for every $y\in X$,
	\[
		\lim_{n\to\infty} \frac{1}{n}\sum_{i=1}^{n}\Vert T^iy\Vert=0.
	\]

	Now assume that $T$ is mean sensitive.
	Then by Proposition~\ref{prop:mean-sen-eq},
	there exists a residual subset $X_0$ of $X$ such that for every $x\in X_0$
	\[
		\limsup_{n\to\infty} \frac{1}{n}\sum_{i=1}^n \Vert T^i x\Vert =\infty.
	\]
	This ends the proof.
\end{proof}

By the proof of Proposition~\ref{prop:dict-mean-asy}, we have the following consequence.

\begin{coro}\label{coro:hyper-mean-eq}
	Let $T$ be a bounded linear operator on a Banach space $X$.
	If $T$ is hypercyclic then $T$ is  mean equicontinuous if and only if for every $x\in X$,
	\[
		\lim_{n\to\infty}\frac{1}{n}\sum_{i=1}^{n}\Vert T^ix\Vert=0.
	\]
\end{coro}

For a subset $A$ of $\mathbb{N}$, the \emph{upper density} of $A$ is defined by 
\[
\udens(A)=\limsup_{n\to\infty} \frac{ \#(A\cap [1,n])}{n},
\]
and the \emph{lower density} of $A$ is 
\[
\ldens(A)=\liminf_{n\to\infty} \frac{ \#(A\cap [1,n])}{n},
\]
where $\#(\cdot)$  denotes the number of elements of a finite set.

Let $T\colon X\to X$ be a bounded linear operator on a Banach space $X$.
We say that $T$ is \emph{frequently hypercyclic} provided that there exists a vector
$x\in X$, called a \emph{frequently hypercyclic vector} for $T$, such that for any non-empty open subset $U$ of $X$,
$\{n\in\mathbb{N}\colon T^nx\in U\}$ has positive lower density, and 
\emph{$\mathcal{U}$-frequently hypercyclic} provided that there exists a vector $x\in X$, called a \emph{$\mathcal{U}$-frequently hypercyclic vector} for $T$, such that such that for any non-empty open subset $U$ of $X$,
$\{n\in\mathbb{N}\colon T^nx\in U\}$ has positive upper density. 

\begin{rem}
Assume that $T\colon X\to X$ is $\mathcal{U}$-frequently hypercyclic.
For any $\mathcal{U}$-frequently hypercyclic vector $x\in X$, one has
\[
    \limsup_{n\to\infty}\frac{1}{n}\sum_{i=1}^n\Vert T^i x\Vert \geq 
     \limsup_{n\to\infty} \frac{1}{n} \#(\{1\leq i\leq n\colon T^i x\in X\setminus \{y\in X\colon \Vert y\Vert \leq 1\})>0.
\]
By Proposition~\ref{prop:dict-mean-asy}, there exists a residual subset $X_0$ of $X$ such that for every $x\in X_0$,
	\[
		\limsup_{n\to\infty} \frac{1}{n}\sum_{i=1}^n \Vert T^i x\Vert =\infty.
\]
\end{rem}

Now we consider the case (3) in Theorem~\ref{thm:trich-hypercity} and its opposite side.

\begin{prop}  \label{prop:dict-mean-proc}
	Let $T$ be an operator on a separable Banach space $X$.
	If $T$ is hypercyclic, then either there exists a residual subset $X_0$ of $X$ such that for every $x\in X_0$
	\[
		\liminf_{n\to\infty} \frac{1}{n}\sum_{i=1}^n \Vert T^i x\Vert =0,
	\]
	or  for every hypercyclic vector $x\in X$
	\[
		\lim_{n\to\infty} \frac{1}{n}\sum_{i=1}^n \Vert T^i x\Vert =\infty.
	\]
\end{prop}

\begin{proof}
	Assume that there exists a hypercyclic vector $x\in X$ such that
	\[
		\liminf_{n\to\infty} \frac{1}{n}\sum_{i=1}^n \Vert T^i x\Vert <\infty.
	\]
	Let
	\[
		M=\liminf_{n\to\infty} \frac{1}{n}\sum_{i=1}^n \Vert T^i x\Vert.
	\]
	Then for every $m\in\mathbb{N}$,
	\[
		\liminf_{n\to\infty} \frac{1}{n}\sum_{i=1}^n \Vert T^i(T^m x)\Vert=M.
	\]
	For every $k\in \mathbb{N}$, let
	\[
		X_k=\biggl\{ y\in X\colon \liminf_{n\to\infty} \frac{1}{n}\sum_{i=1}^n \Vert T^iy\Vert\leq \tfrac{M}{k}\biggr\}.
	\]
	It is easy to see that
	\[
		X_k = \bigcap_{N=1}^\infty
		\biggl\{ y\in X\colon \exists n\geq N\text{ s.t. } \frac{1}{n}\sum_{i=1}^n \Vert T^iy\Vert< \tfrac{M}{k}+\tfrac{1}{N}\biggr\}.
	\]
	Then $X_k$ is a $G_\delta$ subset of $X$.
	As $x$ is hypercyclic, $\{T^m x\colon m\in \mathbb{N}\}$ is dense in $X$.
	Note that $\{\frac{1}{k}T^m x\colon m\in \mathbb{N}\}\subset X_k$.
	Therefore $X_k$ is dense in $X$.
	As
	\[
		X_0:=\biggl\{ y\in X\colon \liminf_{n\to\infty} \frac{1}{n}\sum_{i=1}^n \Vert T^iy\Vert=0\biggr\}=\bigcap_{k=1}^\infty X_k.
	\]
	we have that $X_0$ is residual in $X$.
\end{proof}

Now combining Propositions~\ref{prop:dict-mean-asy} and~\ref{prop:dict-mean-proc}, we get Theorem~\ref{thm:trich-hypercity}.

\section{Proof of Theorem~\ref{thm:SOT-dense} and some examples}

We say that a bounded linear operator $T$ on a Banach space $X$ satisfies the \emph{Kitai criterion}
if there exist two dense subsets $E$ and $F$ of $X$ and a map $S\colon F\to F$ such that
$TSy=y$, $T^kx\to 0$ and $S^ky\to 0$ as $k\to\infty$ for any $x\in E$ and $y\in F$,
and a vector subset $Y$ of $X$ is an \emph{absolutely mean irregular manifold}
if every non-zero vector in $Y$ is absolutely mean irregular for $T$.

\begin{thm}[{\cite{BBP2020}*{Theorem 29}}] \label{thm:dense-mean-asy-vector}
Let $T$ be a bounded linear operator on a separable Banach space $X$.
If 
\[
\Bigl\{x\in X\colon \lim_{n\to\infty}\frac{1}{n}\sum_{i=1}^{n}\Vert T^ix\Vert=0\Bigr\}
\]
 is dense in $X$,
then either
\begin{enumerate}
    \item for every $x\in X$, $\displaystyle \lim_{n\to\infty}\frac{1}{n}\sum_{i=1}^{n}\Vert T^ix\Vert=0$, or
    \item $T$ admits a dense absolutely mean irregular manifold.
\end{enumerate}
\end{thm}

\begin{coro}\label{coro:Kitai-criterion}
Let $T$ be a bounded linear operator on a separable Banach space $X$.
If $T$ satisfies the Kitai criterion and has a non-trivial periodic point then
$T$ admits a dense absolutely mean irregular manifold.
\end{coro}
\begin{proof}
Let $E$ be the dense subset of $X$ in the definition of the Kitai criterion.
For every $x\in E$, $\displaystyle \lim_{n\to\infty}\Vert T^n x\Vert  =0$, then
$\displaystyle \lim_{n\to\infty}\frac{1}{n}\sum_{i=1}^{n}\Vert T^ix\Vert=0$.
As $T$ has a non-trivial periodic point, the case (1) in Theorem \ref{thm:dense-mean-asy-vector} can not happen.
Then $T$ admits a dense absolutely mean irregular manifold.
\end{proof}

By the proof of Theorem~\ref{thm:trich-hypercity}, it is easy to see that the collection of 
absolutely mean irregular vectors is $G_\delta$.
If an operator admits a dense absolutely mean irregular manifold, then the collection of absolutely mean irregular vectors is
residual.
Therefore, Theorem~\ref{thm:SOT-dense} is a directly consequence of the following result.
\begin{thm}
	If $X$ is an infinite-dimensional separable Banach space,
	then the collection of hypercyclic operators admitting a dense absolutely mean irregular manifold is dense in  $\mathcal{L}(X)$ with respect to the strong operator topology.
\end{thm}

\begin{proof}
By the proof of Theorem 1.1 in \cite{S2008}, there exists a bounded linear operator $I+T$ on $X$ satisfying the Kitai criterion.
Then $I+T$ is mixing (see e.g. \cite{GP2011}*{Theorem 3.4}).
In particular, $I+T$ is hypercyclic.
Let $x_0$ be the point in the proof of  Lemma 2.5 in \cite{S2008}.
Then $\Vert x_0\Vert =1$ and $Tx_0=0$, which implies that $x_0$ is a non-trivial fixed point for $I+T$.
By Corollary~\ref{coro:Kitai-criterion},  $I+T$ admits a dense absolutely mean irregular manifold.
By the proof of Theorem 8.18 in \cite{GP2011}, the similarity orbit of $I+T$,
	\[
		\mathcal{S}(I+T)=\{A^{-1}(I+T)A; A \colon X\to X \text{ invertible}\},
	\]
is dense in $\mathcal{L}(X)$ with respect to the strong operator topology.
It is clear that the property of admitting a dense absolutely mean irregular manifold is invariant under similarity.
So the collection of hypercyclic operators with a dense absolutely mean irregular manifold is dense in  $\mathcal{L}(X)$ with respect to the strong operator topology.
\end{proof}

We will need the following result about a sufficient condition for the case (2) in Theorem~\ref{thm:trich-hypercity}. We refer the reader to \cite{BM2009}*{Chapter 5} about ergodic measures of an operator.

\begin{prop}\label{prop:measure-with-finite-first-moment}
	Let $T$ be a bounded linear operator on a separable Banach space $X$.
	If $T$ admits an ergodic measure $\mu$ with full support such that $\int_X \Vert x\Vert d\mu(x)<\infty$, then there is a residual subset of absolutely mean irregular vectors.
\end{prop}
\begin{proof}
	As the function $f\colon X\to \mathbb{R}$, $x\mapsto \Vert x\Vert$, is integrable, by the Birkhoff ergodic theorem (see e.g. \cite{BM2009}*{Theorem 5.3}) for $\mu$-a.e. vector $x\in X$,
	\[
		\lim_{n\to\infty}\frac{1}{n}\sum_{i=1}^{n}\Vert T^ix\Vert=\int_X \Vert z\Vert d\mu(z)>0.
	\]
	As $\mu$ is ergodic with full support, the collection of hypercyclic vectors has full measure.
	In particular, there exists a hypercyclic vector $x\in X$ such that
	\[
		\lim_{n\to\infty}\frac{1}{n}\sum_{i=1}^{n}\Vert T^ix\Vert=\int_X \Vert z\Vert d\mu(z)>0.
	\]
	So cases (1) and (3) in Theorem~\ref{thm:trich-hypercity} can not happen.
	This implies  that there is a residual subset of absolutely mean irregular vectors.
\end{proof}

If in addition $X$ is a separable Hilbert space.
The \emph{strong$^*$ operator topology} on $\mathcal{L}(X)$
is defined as follows:
any $T\in \mathcal{L}(X)$ has a neighborhood basis consisting of sets of the form
\begin{align*}
	V^*_{x_1,\dotsc,x_n,\varepsilon,T}=\{
	S\in \mathcal{L}(X)\colon \|Sx_i-Tx_i\|<\varepsilon, \|S^*x_i-T^*x_i\|<\varepsilon,\ i=1,\dotsc,n\},
\end{align*}
where $x_1,\dotsc,x_n\in X$, $\varepsilon>0$ and $T^*$ is the adjoint operator of $T$.
It is clear that the  strong operator topology is coarser than the strong$^*$ operator topology.
For $M>0$, the closed ball $\mathcal{L}_M(X)=\{T\in \mathcal{L}(X)\colon \|T\|\leq M\}$ with respect to strong operator topology or strong$^*$ operator topology becomes a Polish (separable and completely metrizable) space, see e.g. \cite{P1989}*{Section 4.6.2}.
For the operators on a Hilbert space, we can strengthen Theorem~\ref{thm:SOT-dense} as follows.

\begin{thm} \label{thm:Hilbert-SOT-star}
	Assume that $X$ is an infinite-dimensional separable Hilbert space and $M>1$.
	Then the intersection of the collection of hypercyclic operators with a residual subset of absolutely
	mean irregular vectors and $\mathcal{L}_M(X)$ is a dense $G_\delta$ subset of $\mathcal{L}_M(X)$ with respect to the strong$^*$ operator topology.
\end{thm}

\begin{proof}
	Fix $M>1$. 
 We first show that the collection of hypercyclic operators with a residual subset of absolutely
	mean irregular vectors is a $G_\delta$
	subset of $\mathcal{L}_M(X)$ with respect to the strong operator topology.
 Let $\mathcal{U}$ be a countable topological basis of $X$ consisting of non-empty open subsets.
	For every two sets $U$ and $V$ in $\mathcal{U}$, let
	\[
		\mathcal{T}(U,V)=\{T\in \mathcal{L}(X)\colon
		\exists n\in\mathbb{N} \text{ s.t. } T^nU\cap V\neq\emptyset\}
	\]
	For every $n\in\mathbb{N}$, the map $\mathcal{L}_M(X)\to \mathcal{L}_M(X)$, $T\mapsto T^n$ is continuous with respect to the strong operator topology, see e.g. \cite{GMM2021}*{Lemma 2.1}.
	Then $\mathcal{T}_M(U,V)$ is an open subset of $\mathcal{L}_M(X)$.
	By the Birkhoff transitivity theorem (see e.g. \cite{GP2011}*{Theorem 1.16}), the collection of hypercyclic operators is the set
	\[
		\bigcap_{U,V\in\mathcal{U}} \mathcal{T}(U,V),
	\]
	which is a $G_\delta$ subset  of $\mathcal{L}_M(X)$.

	For every set $U$ in $\mathcal{U}$ and $N\in \mathbb{N}$, let
	\begin{align*}
		\mathcal{M}(U,N)=\biggl\{T\in \mathcal{L}_M(X)\colon &
		\exists n>N \text{ and } x\in U \text{ s.t. }
		\frac{1}{n}\sum_{i=1}^n \Vert T^ix\Vert<\frac{1}{N},\\
		 & \exists m>N \text{ and } y\in U \text{ s.t. }
		\frac{1}{m}\sum_{i=1}^m \Vert T^iy\Vert>N\biggl\}.
	\end{align*}
	Then $\mathcal{M}(U,N)$ is an open subset of $\mathcal{L}_M(X)$.
	It is not hard to check that the collection of operators with a residual subset of absolutely mean irregular vectors is the set
	\[
		\bigcap_{U\in\mathcal{U},N\in\mathbb{N}} \mathcal{M}(U,N),
	\]
	which is a $G_\delta$ subset of $\mathcal{L}_M(X)$. 
 
	Since the strong$^*$ operator topology is finer than the strong operator topology, this collection is  also  a $G_\delta$
	subset of $\mathcal{L}_M(X)$ with respect to  the strong$^*$ operator topology.
	By \cite{GMM2021}*{Corollary 2.12} the collection $\textrm{G-MIX}(X)$ of operators in $\mathcal{L}_M(X)$ which admit a mixing invariant measure in the Gaussian sense with full support is dense in $\mathcal{L}_M(X)$ with respect to  the strong$^*$ operator topology.
	By Fernique's integrability theorem (see e.g. \cite{BM2009}*{Exercise 5.5}),
	every Gaussian measure $\mu$ has finite moments of all orders, in particular, $\int_X \Vert x\Vert d\mu(x)<\infty$.
	According to Proposition~\ref{prop:measure-with-finite-first-moment},
	every operators in $\textrm{G-MIX}(X)$ admits a residual subset of absolutely mean irregular vectors. This end the proof.
\end{proof}

To conclude this paper, we provide some examples for each case of Theorem~\ref{thm:trich-hypercity}. Following \cite{LH2015}, we say that an operator $T\colon X\to X$
is \emph{absolutely Ces\'aro bounded} if
there exists a constant $C>0$ such that
\[
	\sup_{n\in\mathbb{N}}\frac{1}{n}\sum_{i=1}^n\Vert T^ix\Vert \leq C\Vert x\Vert
\]
for all $x\in X$.
It is shown in \cite{JL2024}*{Theorem 4.32} that
an operator is absolutely Ces\'aro bounded if and only if it mean equicontinuous.

\begin{exam}
	According to \cite{LH2015}*{Theorem 3.5 and Remark 3.6}, for each $1\leq p<\infty$
	there exists a mixing unilateral weighted  backward shift operator $T$ on $ \ell^p(\mathbb{N})$ which is absolutely Ces\'aro bounded.
	As it is also mean equicontinuous,
	by Corollary \ref{coro:hyper-mean-eq} for any $x\in \ell^p(\mathbb{N})$, one has
	\[
		\lim_{n\to\infty}\frac{1}{n}\sum_{i=1}^{n}\Vert T^ix\Vert=0.
	\]
 This operator $T$ satisfies the case (1) in the Theorem~\ref{thm:trich-hypercity}.
\end{exam}

By Theorems~\ref{thm:SOT-dense} and \ref{thm:Hilbert-SOT-star}, there are plenty of hypercyclic operators which satisfy the case (2) in the Theorem~\ref{thm:trich-hypercity}. 
The following example in \cite{BBP2020} is more subtle.

\begin{exam}
	According to  \cite{BBP2020}*{Theorem 37},
	for each $1\leq p<\infty$ there exists a bilateral shift $T$ on $\ell^p(v,\mathbb{Z})$ such that $T$ is hypercyclic and completely absolutely mean irregular, that is every non-zero vector is absolutely mean irregular.
\end{exam}

\begin{exam}
According to \cite{BR2015}*{Theorem 7}, there exists a frequently hypercyclic weighted shift $T$ on $c_0(\mathbb{Z})$ that is not distributionally chaotic.
In fact, it is proved in subsection 6.3 of \cite{BR2015} that the orbit
of every non-zero vector is not distributionally near $0$, that is for every non-zero vector $x\in X$ there exists $\varepsilon=\varepsilon(x)>0$ such that 
\[
    \ldens(\{n\in\mathbb{N} \colon \Vert T^n x\Vert \geq \varepsilon\})>0.
\]
Then for every non-zero vector $x\in X$,
\[
  \liminf_{n\to\infty}\frac{1}{n}\sum_{i=1}^n\Vert T^i x\Vert \geq  \varepsilon \ldens(\{n\in\mathbb{N} \colon \Vert T^n x\Vert \geq \varepsilon\})>0. 
\]
This shows that $T$ satisfies the case (3) in the Theorem~\ref{thm:trich-hypercity}.
\end{exam}

If one does not require the frequent hypercyclicity, 
there are relatively simple examples for  the case (3) in the Theorem~\ref{thm:trich-hypercity} as follows.

\begin{prop}
	There exists a hypercyclic weighted backward shift $B_w$ on $\ell^p(\mathbb{Z})$ with $1\leq p<\infty$
	such that for every non-zero vector $x\in \ell^p(\mathbb{Z})$,
	\[
		\lim_{n\to\infty} \frac{1}{n}\sum_{i=1}^n\Vert B_w^i x\Vert =\infty.
	\]
\end{prop}
\begin{proof}
	Let $a_i=2^{2+i}$, $b_i=a_i+i-1+i$, and $c_i=b_i+2i$ for $i\in\mathbb{N}$.
	Define a weight $w=(w_n)_{n\in\mathbb{Z}}$ as follows:
	\[
		w_n = \begin{cases}
			2,           & n> 0;                                                 \\
			\frac{1}{2}, & a_i+1\leq -n\leq b_i \text{ for some }i\in\mathbb{N}; \\
			2,           & b_i+1\leq -n\leq c_i \text{ for some }i\in\mathbb{N}; \\
			1,           & \text{otherwise},
		\end{cases}
	\]
	The product $\prod^0_{j=-n}w_j$ is decreasing for $-n$ in $[a_i+1,b_i]$ and increasing for $-n$ in  $[b_i+1,a_{i+1}]$.
	In particular, for every $i\in\mathbb{N}$,
	\[
		\prod^0_{j=-b_i}w_j=\frac{1}{2^i}, \text{ and }
		\prod^0_{j=-n}w_j=2^i \text{ for }  c_i\leq -n\leq a_{i+1}.
	\]
	Let $B_w$ be the  weighted backward shift on $\ell^p(\mathbb{Z})$  with the weight $w$.
	It is clear that
	\[
		\lim_{i\to +\infty} \prod^0_{j=-b_i}w_j=0,\quad \lim_{n\to +\infty} \prod^{n}_{j=1}w_j=\infty.
	\]
	As $\frac{1}{2}\leq |w_j|\leq 2$ for all $j\in\mathbb{Z}$, for any $k\in\mathbb{Z}$,
	\[
		\lim_{i\to +\infty} \prod^k_{j=k-b_i}w_j=0\quad
		\lim_{i\to +\infty} \prod^{b_i+k}_{j=k+1}w_j=\infty.
	\]
	By the hypercyclic characterization of  weighted backward shifts (see e.g. \cite{GP2011}*{Example 4.5}),  $B_w$ is hypercyclic.

	Note that
	\begin{align*}
		\lim_{n\to\infty} \frac{1}{n}\sum_{i=1}^n\Vert B_w^i e_0\Vert
		 & =\lim_{n\to\infty} \frac{1}{n}\sum_{i=1}^n \prod_{j=-i+1}^0 w_j \\
		 & \geq
		\lim_{i\to\infty} \frac{1}{c_i}  2^{i-1}(a_{i}-c_{i-1})=\infty.
	\end{align*}
	For all $j\in\mathbb{Z}$, as there exists a positive constant $\alpha=\alpha(j)$ such that $e_j=\alpha B_w^{-j} e_0$, we have
	\begin{align*}
		\lim_{n\to\infty} \frac{1}{n}\sum_{i=1}^n\Vert B_w^i e_j\Vert =
		\lim_{n\to\infty} \frac{1}{n}\sum_{i=1}^n\Vert B_w^i (\alpha B_w^{-j} e_0)\Vert =\infty.
	\end{align*}
	For every $x=(x_n)_{n\in\mathbb{Z}}\in\ell^p(\mathbb{Z})$, if $x\neq 0$,
	there exists $j\in \mathbb{Z}$ such that $x_j\neq 0$. Then
	\begin{align*}
		\lim_{n\to\infty} \frac{1}{n}\sum_{i=1}^n\Vert B_w^i x\Vert \geq
		\lim_{n\to\infty} \frac{1}{n}\sum_{i=1}^n\Vert B_w^i (x_je_j)\Vert =\infty.
	\end{align*}
	This ends the proof.
\end{proof}


\noindent \textbf{Acknowledgment.}
The author was supported in part by NSF of China (Grant nos.~12222110 and 12171298)
and would like to thank Professors Quentin Menet and Xinseng Wang for their helpful comments.
The author also expresses many thanks to the anonymous
referee, whose suggestions have substantially improved this paper.


\begin{bibsection}

	\begin{biblist}

		\bib{BM2009}{book}{
		author={Bayart, Fr\'{e}d\'{e}ric},
		author={Matheron, \'{E}tienne},
		title={Dynamics of linear operators},
		series={Cambridge Tracts in Mathematics},
		volume={179},
		publisher={Cambridge University Press, Cambridge},
		date={2009},
		pages={xiv+337},
		isbn={978-0-521-51496-5},
		review={\MR{2533318}},
		doi={10.1017/CBO9780511581113},
		}

  \bib{BR2015}{article}{
   author={Bayart, Fr\'{e}d\'{e}ric},
   author={Ruzsa, Imre Z.},
   title={Difference sets and frequently hypercyclic weighted shifts},
   journal={Ergodic Theory Dynam. Systems},
   volume={35},
   date={2015},
   number={3},
   pages={691--709},
   issn={0143-3857},
   review={\MR{3334899}},
   doi={10.1017/etds.2013.77},
}
        
		\bib{B1988}{book}{
			author={Beauzamy, Bernard},
			title={Introduction to operator theory and invariant subspaces},
			series={North-Holland Mathematical Library},
			volume={42},
			publisher={North-Holland Publishing Co., Amsterdam},
			date={1988},
			pages={xiv+358},
			isbn={0-444-70521-X},
			review={\MR{0967989}},
		}

		\bib{BBMP2011}{article}{
		author={Berm\'{u}dez, T.},
		author={Bonilla, A.},
		author={Mart\'{\i}nez-Gim\'{e}nez, F.},
		author={Peris, A.},
		title={Li-Yorke and distributionally chaotic operators},
		journal={J. Math. Anal. Appl.},
		volume={373},
		date={2011},
		number={1},
		pages={83--93},
		issn={0022-247X},
		review={\MR{2684459}},
		doi={10.1016/j.jmaa.2010.06.011},
		}
        
		\bib{BBP2020}{article}{
			author={Bernardes, N. C., Jr.},
			author={Bonilla, A.},
			author={Peris, A.},
			title={Mean Li-Yorke chaos in Banach spaces},
			journal={J. Funct. Anal.},
			volume={278},
			date={2020},
			number={3},
			pages={108343, 31},
			issn={0022-1236},
			review={\MR{4030290}},
			doi={10.1016/j.jfa.2019.108343},
		}

		\bib{BBPW2018}{article}{
			author={Bernardes, N. C., Jr.},
			author={Bonilla, A.},
			author={Peris, A.},
			author={Wu, X.},
			title={Distributional chaos for operators on Banach spaces},
			journal={J. Math. Anal. Appl.},
			volume={459},
			date={2018},
			number={2},
			pages={797--821},
			issn={0022-247X},
			review={\MR{3732556}},
			doi={10.1016/j.jmaa.2017.11.005},
		}


		\bib{GMM2021}{article}{
		author={Grivaux, S.},
		author={Matheron, \'{E}.},
		author={Menet, Q.},
		title={Linear dynamical systems on Hilbert spaces: typical properties and
				explicit examples},
		journal={Mem. Amer. Math. Soc.},
		volume={269},
		date={2021},
		number={1315},
		pages={v+147},
		issn={0065-9266},
		isbn={978-1-4704-4663-5; 978-1-4704-6468-4},
		review={\MR{4238631}},
		doi={10.1090/memo/1315},
		}

		\bib{GP2011}{book}{
			author={Grosse-Erdmann, K.-G.},
			author={Peris Manguillot, A.},
			title={Linear chaos},
			series={Universitext},
			publisher={Springer, London},
			date={2011},
			pages={xii+386},
			isbn={978-1-4471-2169-5},
			review={\MR{2919812}},
			doi={10.1007/978-1-4471-2170-1},
		}
  
		\bib{JL2024}{article}{
			author={Jiang, Zhen},
			author={Li, Jian},
			title={Chaos for endomorphisms of completely metrizable groups and linear operators on Fr\'echet spaces},
			date={2024},
			note={arXiv:2212.06304v2},
		}


		\bib{LTY2015}{article}{
			author={Li, Jian},
			author={Tu, Siming},
			author={Ye, Xiangdong},
			title={Mean equicontinuity and mean sensitivity},
			journal={Ergodic Theory Dynam. Systems},
			volume={35},
			date={2015},
			number={8},
			pages={2587--2612},
			issn={0143-3857},
			review={\MR{3456608}},
			doi={10.1017/etds.2014.41},
		}

		\bib{LH2015}{article}{
			author={Luo, Lvlin},
			author={Hou, Bingzhe},
			title={Some remarks on distributional chaos for bounded linear operators},
			journal={Turkish J. Math.},
			volume={39},
			date={2015},
			number={2},
			pages={251--258},
			issn={1300-0098},
			review={\MR{3311688}},
			doi={10.3906/mat-1403-41},
		}
        
		\bib{S2008}{article}{
			author={Shkarin, Stanislav},
			title={The Kitai criterion and backward shifts},
			journal={Proc. Amer. Math. Soc.},
			volume={136},
			date={2008},
			number={5},
			pages={1659--1670},
			issn={0002-9939},
			review={\MR{2373595}},
			doi={10.1090/S0002-9939-08-09179-X},
		}

		\bib{P1989}{book}{
			author={Pedersen, Gert K.},
			title={Analysis now},
			series={Graduate Texts in Mathematics},
			volume={118},
			publisher={Springer-Verlag, New York},
			date={1989},
			pages={xiv+277},
			isbn={0-387-96788-5},
			review={\MR{0971256}},
			doi={10.1007/978-1-4612-1007-8},
		}
  
	\end{biblist}
 
\end{bibsection}
\end{document}